\documentclass[11pt]
{amsart}
\usepackage[margin=1.4in]{geometry}

\usepackage{amsmath, amssymb, amsfonts, amsthm, enumerate, amsthm, hyperref}

\newtheorem{theorem}{\bf Theorem}[section]

\newtheorem{lemma}[theorem]{\bf Lemma}
\newtheorem{conjecture}[theorem]{\bf Conjecture}

\newtheorem{corollary}[theorem]{\bf Corollary}

\newcommand{\lemref}[1]{\hyperref[#1]{Lemma~\ref*{#1}}}

\let\theta=\vartheta
\let\sigma=\varsigma

\begin{document}

\title[New bounds on the number of $n$-queens configurations]{New bounds on the number of $n$-queens configurations}

\author[Z.~Luria]{Zur Luria}
 \thanks{ Research supported by Dr.~Max R\"ossler, the Walter Haefner Foundation and the ETH Foundation.} 
\address{Institute of Theoretical Studies, ETH, 8092 Zurich, Switzerland.}
\email{zluria@gmail.com}

\date{\today}

\begin{abstract}
In how many ways can $n$ queens be placed on an $n \times n$ chessboard so that no two queens attack each other? This is the famous $n$-queens problem. Let $Q(n)$ denote the number of such configurations, and let $T(n)$ be the number of configurations on a toroidal chessboard. We show that for every $n$ of the form $4^k+1$, $T(n)$ and $Q(n)$ are both at least $n^{\Omega(n)}$. This result confirms a conjecture of Rivin, Vardi and Zimmerman for these values of $n$ \cite{RVZ94}. We also present new upper bounds on $T(n)$ and $Q(n)$ using the entropy method, and conjecture that in the case of $T(n)$ the bound is asymptotically tight. Along the way, we prove an upper bound on the number of perfect matchings in regular hypergraphs, which may be of independent interest.
\end{abstract}

\maketitle

\section{Introduction}
\label{sec:intro}

The $n$-queens problem has a long history. It was proposed by Max Bezzel in 1848, and subsequently studied by Gauss, P\'{o}lya, and many other prominent mathematicians.
Polya also considered the related problem of placing $n$ nonattacking queens on a torus, in which opposite sides of the board are identified. 
In this paper we are concerned with asymptotic bounds on the number of solutions for these two problems.

Although it is usually classified as mathematical recreation, in fact the $n$-queens problem has a number of applications in computer science and mathematics. In particular, it is of interest as an example of a constraint satisfaction problem, and it is conceivable that a better understanding of the problem could lead to general results on constraint satisfaction problems in general. Indeed, some the results in this paper have this flavor.

The $n$-queens problem on the torus has an elegant algebraic formulation. If we index the rows and columns of an $n \times n$ matrix by elements of $\mathbb{Z}_n$, each diagonal in the toroidal board is a set of the form $\{(x,y):x+y=c\}$ or $\{(x,y):x-y=c\}$, for some $c \in \mathbb{Z}_n$. Thus, a solution is a set $\{(x_i,y_i)\}_{i=1}^n$ such that $$\{x_i\}_i=\{y_i\}_i=\{x_i+y_i\}_i=\{x_i-y_i\}_i=[n].$$ 

It was established by Pauls \cite{P1874_1,P1874_2} that there is a solution to the classical $n$-queens problem for every $n \geq 4$. P\'{o}lya \cite{Pol18} proved that there is a solution in the toroidal case if and only if $n$ is not divisible by $2$ or $3$. As far as we are aware, the only upper bound on $Q(n)$ and $T(n)$ is the trivial bound $n!$. Rivin, Vardi and Zimmerman proved in \cite{RVZ94} that $T(n) \geq 2^{\frac{n}{5}}$ and that $Q(n) \geq 4^{\frac{n}{5}}$. They conjectured that in fact 
$$
\lim_{\substack{n \rightarrow \infty \\ (n,6)=1}}{\frac{\log\left(T(n)\right)}{n \log n}} > 0 \text{ , } \lim_{n \rightarrow \infty }{\frac{\log\left(Q(n)\right)}{n \log n}} > 0.
$$
For a thorough survey on the $n$-queens problem and its applications, see \cite{BS09}.

We prove the following results.
\begin{theorem} \label{thm:lower_bound}
Let $k \in \mathbb{N}$ be a positive integer, and let $n=2^{2k}+1$. Then $T(n) \geq n^{\Omega(n)}$.
\end{theorem}
Note that since every legal toroidal configuration is also a solution for the classical case, this theorem implies the same lower bound on $Q(n)$.
The following theorem provides a nontrivial upper bound on $T(n)$.
\begin{theorem} \label{thm:upper_bound_torus}
For any integer $n$, $T(n) \leq \left((1+o(1)) \cdot \frac{n}{e^3}\right)^{n}$.
\end{theorem}
In fact, this result is obtained as a corollary of Theorem \ref{thm:hypergraphs}, which is a general theorem on hypergraphs, and may be of independent interest.

Theorem \ref{thm:hypergraphs} is obtained via an entropy proof, which has given asymptotically tight results in similar situations \cite{Rad97, LL13, Ke15, GL16}.
These results are the justification for the following conjecture, which is a strengthening of Conjecture 1 from \cite{RVZ94}.
\begin{conjecture}
For any integer $n$ that is not divisible by $2$ or $3$ there holds $T(n)=\left((1+o(1)) \cdot \frac{n}{e^3}\right)^{n}$.
\end{conjecture}

Note that Theorem \ref{thm:upper_bound_torus} does not imply an upper bound on the number of solutions in the classical case. 
\begin{theorem} \label{thm:upper_bound_classical}
There exists a constant $\alpha>1.587$ such that for any integer $n$, $Q(n) \leq \left((1+o(1)) \cdot \frac{n}{e^{\alpha}}\right)^{n}$.
\end{theorem}

In section \ref{sec:lower_bound} we prove Theorem \ref{thm:lower_bound}. Theorem \ref{thm:upper_bound_torus} is proven in section \ref{sec:upper_bound_torus}, and Theorem \ref{thm:upper_bound_classical} is proven in section \ref{sec:upper_bound_classical}.

\section{A lower bound}
\label{sec:lower_bound}

The proof idea is similar to a recent lower bound by Potapov \cite{Pot16} on the number of Latin hypercubes. In essence, the idea is to find a ``flexible" starting configuration $B$, which admits many different possibilities for local modification. We show that one can obtain at least $n^{\Omega(n)}$ different configurations.

As detailed in the introduction, we index the rows and columns by elements of $\mathbb{Z}_n$, and set $B = \{(x,y):x=2^k \cdot y\}$. 

\begin{lemma} 
$B$ is a legal toroidal $n$-queens configuration.
\end{lemma}
\begin{proof}
First, we note that the elements $(2^k - 1),2^k$, and  $(2^k + 1)$ are all invertible in $\mathbb{Z}_n$. Indeed, $2^{2k} = -1$ modulo $n$, and therefore $2^k \cdot 2^{3k} = 1$ modulo $n$. 
We also have $(2^k - 1)(2^k+1) = -2$ modulo $n$, and so $\left((2^k - 1)(2^k+1)\right)^{2k} = -1$ modulo $n$. Hence, $\left((2^k - 1)(2^k+1)\right)^{4k} = 1$ modulo $n$, implying that $(2^k - 1)$ and $(2^k + 1)$ are invertible.

Clearly, for every $y \in \mathbb{Z}_n$ there is a unique $x \in \mathbb{Z}_n$ such that $(x,y) \in B$. The fact that $2^k$ is invertible implies that for every $x \in \mathbb{Z}_n$ there is a unique $y \in \mathbb{Z}_n$ such that $(x,y) \in B$. It remains to show that each diagonal contains a unique element of $B$.

In other words, we must show that for every $z \in \mathbb{Z}_n$ there is a unique $(x,y) \in B$ such that $x+y=z$, and a unique $x',y' \in B$ such that $x'-y'=z$. The first statement is equivalent to the fact that the equation $2^k y + y =z$ has a unique solution, which follows from the invertibility of $(2^k+1)$. Likewise, the second statement follows from the invertibility of $(2^k-1)$.
\end{proof}

Given two elements $(x_1,y_1),(x_2,y_2)$ in a permutation matrix, it is always possible to obtain a new permutation matrix by ``flipping" them: That is, replacing them with the elements $(x_1,y_2),(x_2,y_1)$. We would like to define a similar flipping operation for $n$-queens configurations, but the problem is that the diagonal constraints would be violated. However, it turns out that $B$ has the wonderful property that for any pair of elements in $B$, there is an additional pair, such that we can flip both pairs together without violating any diagonal constraints. This property will enable us to make local changes to $B$.

\begin{lemma} 
\label{lem:pair}
For every pair of elements $(x_1,y_1),(x_2,y_2) \in B$, there is a unique pair $(x_3,y_3),(x_4,y_4) \in B$ such that
\begin{itemize}
\item $x_1 + y_1=x_3+y_4$. 
\item $x_2 + y_2=x_4+y_3$. 
\item $x_3 + y_3=x_2+y_1$. 
\item $x_4 + y_4=x_1+y_2$. 
\item $x_3 - y_3=x_1-y_2$. 
\item $x_2 - y_2=x_3-y_4$. 
\item $x_1 - y_1=x_4-y_3$. 
\item $x_4 - y_4=x_2-y_1$. 
\end{itemize}

\end{lemma}
\begin{proof}
Note that for a fixed pair $(x_1,y_1),(x_2,y_2) \in B$, we have eight equations and only two degrees of freedom. We rewrite the equations, recalling that $x_i = 2^k y_i$. We have
\begin{itemize}
\item $(2^k+1) y_1=2^k y_3+y_4$. 
\item $(2^k+1) y_2=2^k y_4+y_3$. 
\item $(2^k+1) y_3=2^k y_2+y_1$. 
\item $(2^k+1) y_4=2^k y_1+y_2$.
\item $(2^k-1) y_3=2^k y_1-y_2$. 
\item $(2^k-1) y_2=2^k y_3-y_4$. 
\item $(2^k-1) y_1=2^k y_4-y_3$.  
\item $(2^k-1) y_4=2^k y_2-y_1$. 
\end{itemize}
Noting that $2^k \cdot (2^k + 1) = 2^k - 1$ modulo $n$, it follows that the last four equations are $2^k$ times the first four. The third and fourth equations imply that if there is a solution, it is $y_3 = (2^k+1)^{-1} \cdot (2^k y_2 + y_1)$ and $y_4 = (2^k+1)^{-1} \cdot (2^k y_1 + y_2)$. Plugging these values into the first two equations, it is easy to verify that they are satisfied as well. 
\end{proof}

Now, given any pair $(x_1,y_1),(x_2,y_2) \in B$ and $(x_3,y_3),(x_3,y_3) \in B$ satisfying Lemma \ref{lem:pair}, we can locally modify $B$ by setting
\[
 B' = B \setminus \{(x_1,y_1),(x_2,y_2),(x_3,y_3),(x_4,y_4)\} \cup \{(x_1,y_2),(x_2,y_1),(x_3,y_4),(x_4,y_3)\}.
\]
The equations in Lemma \ref{lem:pair} ensure that $B'$ is a legal $n$-queens configuration. For example, the diagonal whose elements sum to $x_1+y_1$ will lose the element $(x_1,y_1)$, but gain the element $(x_3,y_4)$. We call such a modification a flip.

As a corollary, note that for every square $(x,y)$ on the board not containing a queen, there is a \textit{unique} flip that places a queen on that square. This flip is easy to find, since it must involve the queens on the same row and column as $(x,y)$.

Now, since each unoccupied square defines a unique flip, and each flip places a queen on four unoccupied squares, the total number of possible flips in $B$ is $\frac{n(n-1)}{4}$.

We say that two flips in $B$ are disjoint if they do not involve any of the same queens. One can perform any number of flips simultaneously and independently, provided that they are all disjoint. Moreover, this process is reversible: Since each queen that is not in its original location must be the result of a unique flip, we can reconstruct which flips were made from the resulting configuration. Therefore, the number of toroidal $n$-queens configurations is at least the number of disjoint sets of flips.

Now, we greedily selecting one flip at at time, taking care that all of the flips are disjoint. As each flip intersects at most $4(n-1)$ other flips, and the total number of flips in $B$ is $\frac{n(n-1)}{4}$, there are at least $\frac{n(n-1)}{4}-(i-1)\cdot 4(n-1)$ ways to select the $i$-th flip. We continue this process for $t:=n/16$ steps, obtaining a collection of $t$ disjoint flips. The number of possible collections that we can construct in this manner is lower bounded by $ \frac{1}{t} \cdot \prod_{i=1}^t{\frac{n(n-1)}{4}-(i-1)\cdot 4(n-1)} $. By applying Stirling's formula and simplifying, this is $\Theta(n)^{n/16}$, completing the proof.

\section{An entropy upper bound for the toroidal case}
\label{sec:upper_bound_torus}

In recent years, the entropy method has been applied to many different problems, and it has proven to be a versatile and powerful tool. We summarize the properties of information entropy that we need below. 

\begin{itemize}
\item The base $e$ entropy of a discrete random variable $X$ is defined by \\
$H(X) := -\sum_{x \in \text{Range($X$)}}{\Pr(X=x) \log(\Pr(X=x))}$. 
\item For any discrete random variable $X$ there holds $H(X) \leq \log(|\text{Range($X$)}|)$, with equality if and only if $X$ is uniformly distributed.
\item The conditional entropy of $X$ given $Y$ is $H(X|Y)= \mathbb{E}_Y [H(X|Y=y)]$. In other words, if one thinks of $X$ conditioned on the event $Y=y$ as a random variable, then $H(X|Y=y)$ is a well defined function of the value of $Y$. The conditional entropy $H(X|Y)$ is the expectation of this function of $Y$.  
\item The chain rule: $$ H(X_1,...,X_n) = \sum_{i=1}^n{H(X_i|X_1, ... , X_{i-1})}.$$
\item If $Y$ is a function of $X$ then $H(X) \geq H(Y)$. In particular, if two sets of random variables encode the same information, then they have the same entropy.
\end{itemize}
For a thorough introduction to information entropy, see \cite{CT12}.

Rather than addressing the $n$-queens problem directly, we feel that it might be beneficial to prove a much more general theorem, from which Theorem \ref{thm:upper_bound_torus} will follow as a corollary.

A $d$-uniform hypergraph is $k$-regular if every vertex belongs to exactly $k$ edges. The codegree of two vertices is the number of edges that contain both. A perfect matching in a hypergraph is a collection of disjoint edges that cover all of the vertices. The following theorem gives an asymptotic upper bound on the number $M(H)$ of perfect matchings in a regular hypergraph $H$ whose codegrees are small. 

\begin{theorem} \label{thm:hypergraphs}
Let $d$ be a constant, and let $k \rightarrow \infty$ as $n \rightarrow \infty$. Let $H=\langle V,E \rangle$ be a $d$-uniform, $k$-regular hypergraph on $n$ vertices such that all of the codegrees are $o(k)$. Then 
\[
M(H) \leq \left((1+o(1))\frac{k}{e^{d-1}}\right)^{n/d}.
\]
\end{theorem}
Since a great many combinatorial problems are equivalent to finding a perfect matching in some regular hypergraph, this theorem has innumerable applications. The examples below are really just the tip of the iceberg.

\begin{corollary}
\label{cor:examples}
\begin{enumerate}
\item The number of $n$-queens solutions on the torus is at most 
\[
\left((1+o(1))\frac{n}{e^3}\right)^{n}.
\]
 \item The number of transversals in an order-$n$ Latin square is at most
 \[
\left((1+o(1))\frac{n}{e^2}\right)^{n}.
\]
 \item The number of order-$n$ Sudoku squares is at most
 \[
\left((1+o(1))\frac{n}{e^3}\right)^{n^2}.
\]
 \item The number of $(n,q,r)$-Steiner systems is at most 
\[
 \left((1+o(1))\frac{\binom{n}{q-r}}{e^{\binom{q}{r}-1}}\right)^{\binom{n}{r}/ \binom{q}{r}}.
 \]
 \item Let $r' < r < q$. The number of $(n,q,r')$-Steiner systems in a given $(n,q,r)$-Steiner system  is at most 
 \[
 \left((1+o(1))\frac{\binom{n}{r-r'}}{e^{\binom{q}{r'}-1}}\right)^{\binom{n}{r'}/ \binom{q}{r'}}.
 \]
 \item The number of decompositions of the complete $q$-uniform hypergraph on $n$ vertices into $(n,q,r)$-Steiner systems is at most
 \[
\left((1+o(1))\frac{\binom{n-r}{q-r}}{e^{\binom{q}{r}}}\right)^{\binom{n}{q}}.
\]
 \item The number of ways to decompose the basic $n$-queens solution $B$ from the proof of Theorem \ref{thm:lower_bound} into $\frac{n}{4}$ disjoint flips is at most
 \[
 \left((1+o(1))\frac{n}{e^{3}}\right)^{n/4}.
 \]
\end{enumerate}
\end{corollary}

Note that the first item of the Corollary is the statement of Theorem \ref{thm:upper_bound_torus}.

\begin{proof}[Proof of Corollary \ref{cor:examples}]
The proof in each case consists of describing an equivalent formulation of the problem as counting perfect matchings in some regular hypergraph with small codegrees, and then plugging the appropriate values for $n$,$k$ and $d$ into Theorem \ref{thm:hypergraphs}.
\begin{enumerate}
\item Each row, column, forward diagonal and backward diagonal will correspond to a vertex, and there is a $4$-uniform edge for each square in the board. Each such edge contains the row, column and diagonals of the corresponding square. It is clear that a perfect matching in this hypergraph corresponds to a toroidal $n$-queens solution. Note that the number of vertices is $4n$, the degree of every vertex is $n$ and the uniformity of the hypergraph is $4$. Moreover, the codegree of every two vertices in the graph is $O(1)$, and so plugging these values into the Theorem gives the result.

 \item This result was previously published in \cite{Tar15, GL16}. An order-$n$ Latin square is equivalent to a $3$-partite, $3$-uniform hypergraph with $n$ vertices in each part such that any two vertices from different parts belong to a unique edge. Now, a transversal in the Latin square corresponds to a perfect matching in this hypergraph. Since it has has $3n$ vertices, uniformity $3$, the degrees are all $n$ and the codegrees are $1$, the Theorem now gives the result.
 
 \item  An order-$n$ Sudoku square is equivalent to a $4$-partite, $4$-uniform hypergraph with $n^2$ vertices in each part. The parts correspond to the pairs (row, column), (column,symbol), (row, symbol) and (box,symbol). We add an edge for each possible symbol and for each location in the $n\times n$ matrix, so that there are $n^3$ edges. Note that an order-$n$ Sudoku square is equivalent to a perfect matching in this hypergraph. Therefore, as it has $4n^2$ vertices, uniformity $4$, the degrees are all $n$ and the codegrees are at most $\sqrt{n}$, we are done. 
 
 \item \label{item:designs} This was the upper bound in \cite{Ke15} for the special case of Steiner systems. An $(n,q,r)$-Steiner system is equivalent to a perfect matching in the so-called auxiliary hypergraph, in which each $r$-set corresponds to a vertex, and each $q$-set $f$ corresponds to a hyperedge containing all of the $r$-sets contained by $f$. Applying the Theorem to this hypergraph gives the result.
 
 \item As far as the authors are aware, this is a new result. It follows from taking a subhypergraph of the hypergraph described in the previous item, in which we have only those hyperedges corresponding to $q$-sets in the given $(n,q,r)$-Steiner system.

\item Here, the reformulation of the problem is a little less straightforward. Note that the number of $(n,q,r)$-Steiner systems that are needed to cover all of the $q$-sets of the complete $q$-uniform hypergraph on $n$ vertices is $\binom{n-r}{q-r}$. Thus, a decomposition is equivalent to a $(2q-r)$-uniform hypergraph $H$ over $(2n-r)$ vertices such that
\begin{itemize}
\item There are $n$ ``base" vertices, corresponding to the original $n$ vertices, and $(n-r)$ ``color" vertices.
\item All hyperedges consist of $q$ base vertices and $(q-r)$ color vertices.
\item Every $q$-set of base vertices belongs to a unique hyperedge. We call such $q$-sets ``base tuples".
\item Every $q$-set consisting of $r$ base vertices and $(q-r)$ color vertex belongs to a unique hyperedge. We call such $q$-sets ``color tuples".
\end{itemize}
By indexing the different parts of a given decomposition by $(q-r)$-element subsets of $[n-r]$, we can define the correspondence between decompositions and hypergraphs by adding the hyperedge $B\cup C$ to the hypergraph iff the $q$-set $A$ belongs to part $B$ in the decomposition.

We now wish to define a hypergraph $H'$ such that a perfect matching in $H'$ corresponds to a decomposition hypergraph as above. We take $H'$ to be the $\left(\binom{q}{r}+1\right)$-uniform hypergraph over $\left(\binom{n}{q} + \binom{n}{r} \cdot \binom{n-q}{q-r} \right)$ vertices in which
\begin{itemize}
\item The vertices correspond to the $q$-sets of $H$ that must be covered. That is, they are the color and base tuples.
\item For each base tuple $f$ and $(q-r)$-set of color vertices $g$, we add a hyperedge to $H'$ consisting of $f$ together with all of the color tuples of the form $g\cup h$ for some $r$-set $h\subset f$.
\end{itemize}
The hypergraph $H'$ has the same role as the auxiliary hypergraph in item \ref{item:designs}. As above, a perfect matching in $H'$ corresponds to a legal decomposition hypergraph $H'$.
Since $H'$ is $\binom{n-r}{q-r}$-regular and all codegrees are sufficiently small, the result follows. 

 \item \label{item:queens} We define our hypergraph by taking the vertices to be the queens, while the hyperedges correspond to the flips. 
\end{enumerate}
\end{proof}

We include item \ref{item:queens} to illustrate the versatility of Theorem \ref{thm:hypergraphs}, and also to give an upper bound on the number of $n$-queens solutions that could conceivably be constructed using the methods of Theorem \ref{thm:lower_bound}. As can be seen, there is a substantial gap between this bound and our upper bound on the number of $n$-queens solutions.

We turn to prove the Theorem.
\begin{proof}[Proof of Theorem \ref{thm:hypergraphs}]
The proof follows the same lines as the proof of Theorem 1 in \cite{LL13}. We estimate the entropy of a perfect matching $X$ chosen uniformly at random. For $v \in V$, let $X_v$ be the edge covering the vertex $v$ in $X$. Now, by the chain rule, for any ordering of the vertices there holds:
\[
H(X) = \sum_{v \in V} H(X_v| X_{v'}, v' \text{ precedes } v).
\]
In particular, we can choose a random ordering and take the expectation of the bound with respect to the ordering. It will be convenient to choose the ordering in a slightly roundabout way.
We choose a random vector of real numbers $\alpha = (\alpha_v)_{v \in V}$, where for each vertex $v$, we have $\alpha_v \sim U[0,1]$ independently from the other vertices. The variables $X_v$ are exposed in order of descending $\alpha_v$.
We have
\[
H(X) = \mathbb{E}_{\alpha} \left[\sum_{v \in V} H(X_v| X_{v'}, v' \text{ precedes } v)\right] 
\]
\[
\leq  \mathbb{E}_{\alpha} \left[\sum_{v \in V} \mathbb{E}_{(X_{v'}, v' \text{ precedes } v)}[H(X_v|X_{v'}=x_{v'}, v' \text{ precedes } v)]\right].
\]
Now, fix an ordering and a given vertex $v$. If some other vertex in $X_v$ precedes $v$, then by the time we get to the vertex $v$, we already know the value of $X_v$. Even if $v$ precedes the other vertices in $X_v$, given the values of previous $X_{v'}$, some of the possible values for $X_v$ are ruled out. In particular, if $f$ is an edge containing $v$ and $v' \in f$, then if any of the vertices in $X_{v'}$ precede $v$, $X_v$ cannot be equal to $f$. We denote by $N_v$ the number of possible values for $v$ given previous choices. 

Using the fact that $H(Y) \leq \log(|\text{Range}(Y)|)$ for any discrete random variable $Y$, we have
\[
H(X) \leq  \mathbb{E}_{X} \left[\sum_{v \in V} \mathbb{E}_{\alpha}[\log(N_v)]\right].
\]
Let $A_v$ denote the event that $v$ is the first vertex in $X_v$ to be exposed. Our next step is to split the expectation over the ordering into two expectations: One over $\alpha_v$, and one over the remaining variables in $\alpha$. Note that given $\alpha_v$, the probability that $v$ precedes any set of $j$ vertices is $\alpha_v^j$. This does not change if we also condition over $A_v$.
\[
H(X) \leq  \mathbb{E}_{X} \left[\sum_{v \in V} \mathbb{E}_{\alpha_v}\mathbb{E}_{\alpha | \alpha_v}[\Pr(A_v) \cdot \log(N_v)]\right]
\]
\[
= \mathbb{E}_{X} \left[\sum_{v \in V} \mathbb{E}_{\alpha_v}[ \alpha_v^{d-1} \cdot \mathbb{E}_{\alpha | \alpha_v}[ \log(N_v)]]\right]
\]
\[
\leq \mathbb{E}_{X} \left[\sum_{v \in V} \mathbb{E}_{\alpha_v}[\alpha_v^{d-1} \cdot \log \left(\mathbb{E}_{\alpha | \alpha_v}[N_v]\right)]\right]
\]
\[
= \mathbb{E}_X \left[ \sum_{v \in V} \int_0^1 \alpha_v^{d-1} \log \left( \mathbb{E}_{\alpha | \alpha_v}[N_v] \right)d \alpha_v \right].
\]

We turn to estimate $\mathbb{E}_{\alpha | \alpha_v}[N_v]$. We say that an edge $f\in E$ is bad if $|f \cap X_v| > 1$ for some $v \in V$. Note that the number of bad edges is at most $\left(\frac{n}{d}\right) \cdot \binom{d}{2} \cdot o(k)$. Denote the set of bad edges by $E^B$, and the set of bad edges containing a vertex $v$ by $E^B_v$.

Now, $\mathbb{E}[N_v] = \sum_{f:v \in f}{\Pr(f \text{ is available})}$. There are two cases: If $f \in \{X_v\} \cup E^B$, we upper bound the probability that $f$ is available by $1$. If $f$ is not $X_v$ or a bad edge, then there are exactly $d(d-1)$ vertices whose exposure would eliminate $f$. These are the vertices of the edges $X_{v'}$, for the vertices $v'\neq v$ of $f$. Note that the edges $X_{v'}$ are disjoint from each other and from $X_v$, because $X$ is a perfect matching. Therefore, the probability that a good edge is available is $\alpha_v^{d(d-1)}$, and we have 
\[
\mathbb{E}[N_v] \leq 1+|E^B_v|+ k \cdot \alpha_v^{d(d-1)}.
\]
Therefore,
\[
H(X) \leq \mathbb{E}_X \left[ \sum_{v \in V} \int_0^1 \alpha_v^{d-1} \log \left( 1+|E^B_v|+ k \cdot \alpha_v^{d(d-1)} \right) d \alpha_v \right]
\]
\[
= \mathbb{E}_X \left[  \int_0^1 \left( \sum_{v \in V} x^{d-1} \log \left( 1+|E^B_v|+ k \cdot x^{d(d-1)} \right)\right) dx \right]
\]
\[
\leq \mathbb{E}_X \left[  \int_0^1  n \cdot x^{d-1} \log \left( 1+ k \cdot x^{d(d-1)}+\left(\frac{1}{n}\right)\cdot\sum_{v \in V}|E^B_v| \right) dx \right]
\]
\[
\leq \mathbb{E}_X \left[  \int_0^1  n \cdot x^{d-1} \log \left( k \cdot x^{d(d-1)}+o(k)\right) dx \right]
\]
\[
= n \cdot \int_0^1 x^{d-1} \log \left( k \cdot x^{d(d-1)}+o(k) \right) dx 
\]
\[
= \frac{n}{d}\cdot\left(\log(k)-(d-1)+o(1)\right).
\]

As $H(X)=\log(M(H))$, this completes the proof.

\end{proof}

\section{An entropy upper bound for the classical problem}
\label{sec:upper_bound_classical}

The $n$-queens problem on the non-toroidal chessboard presents a greater challenge for the entropy method. The issue is that in this case, a square on the board may be attacked by two, three, or four queens, whereas in the toroidal case each square is attacked by exactly four queens. We are able to overcome this obstacle, but we no longer believe that the result is tight.

Proceeding along similar lines to the proof of Theorem \ref{thm:hypergraphs}, we estimate the entropy of a random $n$-queens solution $X = \{(i,X_i): 1 \leq i \leq n\}$. We order the variables $X_i$ by descending order of independent random real numbers $\alpha_i \sim U([0,1])$. Let $N_i$ denote the number of positions in the $i$-th row that are not ruled out by previously seen rows. Here a position is ruled out if it shares a column or a diagonal with a previously exposed queen. As above, one can show that
\[
H(X) \leq \mathbb{E}_X \sum_i \mathbb{E}_{\alpha_i} \left[ \log( \mathbb{E}_{\alpha|\alpha_i}[N_i] )\right].
\]
For each row, there is a unique position that is cannot be ruled out, namely $X_i$. Every other position shares a column and up to two diagonals with a queen, and so it can be ruled out by one, two or three different rows. Let $a_i, b_i$ and $c_i$ be the number of positions in the $i$-th row that can be ruled out by three, two and one different rows respectively. Now, a position in the $i$-th row is \textit{not} ruled out if $\alpha_i > \alpha_j$ for every row $j$ that could rule out that position. Therefore, $  \mathbb{E}_{\alpha|\alpha_i}[N_i] =1 + a_i \alpha_i^3 + b_i \alpha_i^2 + c_i \alpha_i$, and we have
\[
H(X) \leq \mathbb{E}_X \sum_i \int_0^1\log( 1 + a_i x^3 + b_i x^2 + c_i x )dx .
\]

It is natural to guess that the $1$ inside the log does not matter too much. That is the content of the following lemma. 
\begin{lemma}
For any $a,b,c \geq 0$ such that $a+b+c = (n-1)$, there holds
\[\int_0^1\log(1 + a x^3 + b x^2 + c x) dx = \int_0^1 \log(a x^3 + b x^2 + c x) dx + O(n^{-\frac12}) .\]
\end{lemma}
\begin{proof}
For all $0 < x < 1$ there holds
\[ \log(1 + a x^3 + b x^2 + c x) =   \log(a x^3 + b x^2 + c x) + \log(1+1/(a x^3 + b x^2 + c x)).\]
We use the fact that $a + b + c = n-1$.
\[ \log(1+1/(a x^3 + b x^2 + c x)) \leq \log \left(1+\frac{1}{(n-1)x}\right) =\]
\[   \int_0^{n^{-\frac12}} \log \left(1+\frac{1}{(n-1)x}\right) + \int_{n^{-\frac12}}^1 \log \left(1+\frac{1}{(n-1)x}\right) \leq\]
\[  \int_0^{n^{-\frac12}} \log\left(\frac{2}{x \sqrt{n}}\right) dx + \frac{n^{\frac12}}{n-1} .\]
Now,
\[  \int_0^{n^{-\frac12}} \log \left(\frac{2}{x \sqrt{n}}\right) dx \leq - \int_0^{n^{-\frac12}} \log(x) dx = \]
\[
\int_0^{n^{-\frac12}} \log(x) dx = [ x \log(x) - x ] \vert_0^{n^{-\frac12}} \leq  n^{-\frac12}.
\]
This proves the lemma.
\end{proof}

Thus, we have \[ H(X) \leq O(n^{\frac12}) + \mathbb{E}_X \int_0^1 \sum_i \log(a_i x^3 + b_i x^2 + c_i x)dx = \]
\[ O(n^{\frac12}) + n \int_0^1{\log(x)dx} + \mathbb{E}_X \int_0^1 \sum_i \log(a_i x^2 + b_i x + c_i)dx = \]
\[ O(n^{\frac12}) - n + \mathbb{E}_X \int_0^1 \sum_i \log(a_i x^2 + b_i x + c_i)dx \]

To get a meaningful upper bound, we must show that $a_i$ and $b_i$ cannot all be very small. That is the content of the following lemma.

\begin{lemma}
\label{lem:concentric}
\[ \sum_i {2 a_i + b_i} \geq \left( \frac{5}{4} \right)n^2-6n .\]
\end{lemma}
\begin{proof}
Let $D(i,j)$ be the number of positions that share a diagonal with the square $(i,j)$, other than the square $(i,j)$ itself. Note that the sum $\sum_i {2 a_i + b_i}$ counts pairs of the form 
$((i,j) \in X, (i',j') \notin X)$ such that $(i,j)$ shares a diagonal with $(i',j')$. We have 
\[
 \sum_i {2 a_i + b_i} = \sum_{(i,j) \in X} D(i,j).
\] 
The lemma now follows from the observation that $D(i,j)$ is constant along concentric square rings. Concretely, there holds $D(i,j)=(n-3)+ 2 * \min\{ i,j,n+1-i,n+1-j \}$. For example, for $n=5$ the values of $D(i,j)$ are described by the matrix
\[
D = \left( \begin{matrix}
4 & 4 & 4 & 4 & 4 \\
4 & 6 & 6 & 6 & 4 \\
4 & 6 & 8 & 6 & 4 \\
4 & 6 & 6 & 6 & 4 \\
4 & 4 & 4 & 4 & 4
\end{matrix} \right).
\]
Now, as every square ring can contain at most $4$ queens in $X$, we have
\[
 \sum_{(i,j) \in X} D(i,j) \geq \sum_{k=0}^{\lfloor \frac{n}{4}\rfloor-1} {4 \cdot (n - 1 + 2k)} \geq
\]
\[
(n-4)(n-1) + 4 \cdot \left( \frac{n}{4} -1\right)\left(\frac{n}{4} \right) \geq \left( \frac{5}{4} \right) \cdot n^2 - 6n.
\]
\end{proof}

The following lemma uses lemma \ref{lem:concentric} to bound the sum $\sum_i{\log(a_i x^2 + b_i x + c_i)}$.

\begin{lemma}
For any $0 < x < 1$ there holds
\[
\sum_i{\log(a_i x^2 + b_i x + c_i)} \leq O(1) + n \log\left( \left(\frac{5}{8}\right) n x^2 + \left(\frac{3}{8}\right) n \right).
\]
\end{lemma}
\begin{proof}
Let $a'_i = a_i + \frac{b_i}{2}$ and $c'_i = c_i + \frac{b_i}{2}$. Note that 
\[ (a'_i x^2 + c'_i) - (a_i x^2 + b_i x + c_i)  =  \frac{b_i}{2} (x-1)^2 \geq 0, \]
and therefore 
\[
(a'_i x^2 + c'_i) \geq (a_i x^2 + b_i x + c_i)
\]
As $c'_i=(n-1)-a'_i$, we have 
\[
\sum_i{\log(a_i x^2 + b_i x + c_i)} \leq \sum_i{\log((x^2-1) a'_i + (n-1))}. 
\]
Note that the function $f(t) = \log((x^2-1) t + (n-1))$ is concave. Therefore, 
\[
\sum_i{\log((x^2-1) a'_i + (n-1))} \leq n \log( (x^2-1)\left(\frac{\sum_i a'_i}{n} \right) + (n-1)).
\]
Lemma \ref{lem:concentric} implies that $\sum_i{a'_i} \geq \frac{5}{8} n^2 - 3n$, so
\[
n \log( (x^2-1)\left(\frac{\sum_i a'_i}{n} \right) + (n-1)) \leq n \log\left( \left(\frac{5}{8}\right) n x^2 + \left(\frac{3}{8}\right) n + 2 \right) =
\] 
\[
O(1) + n \log\left( \left(\frac{5}{8}\right) n x^2 + \left(\frac{3}{8}\right) n \right).
\]

\end{proof}

We have shown that 
\[
H(X) \leq O(n^{\frac12}) - n + n \cdot \int_0^1 \log\left( \left(\frac{5}{8}\right) n x^2 + \left(\frac{3}{8}\right) n \right) dx =
\]
\[
O(n^{\frac12}) + n \log(n) - n + n \cdot \int_0^1 \log\left( \left(\frac{5}{8}\right) x^2 + \left(\frac{3}{8}\right) \right) dx =
\]
\[
O(n^{\frac12}) + n \log(n) - n + n \cdot (-2 + 2 \sqrt{3/5} \cdot \text{arctan}(\sqrt{5/3})).
\]
Set $\alpha = -3 + 2 \sqrt{3/5} \cdot \text{arctan}(\sqrt{5/3}) \approx -1.587$. We have 
\[
H(X) \leq O(n^{\frac12}) + n (\log(n)-\alpha),
\]
and therefore the number of $n$-queens configurations is at most $ \left((1+o(1)) \frac{n}{e^\alpha}\right)^n$, as desired.

\end{document}